\documentclass[reqno]{amsart}
\usepackage{hyperref}
\usepackage{amsfonts}
\usepackage{amssymb}
\usepackage{amsmath}
\usepackage{enumitem}
\usepackage{amscd,mathtools}
\renewcommand{\div}{\operatorname{div}}
\newcommand{\spn}{\operatorname{span}}
\newcommand{\eucl}{\operatorname{eucl}}
\newcommand{\ip}[2]{\ensuremath{\langle #1 , #2 \rangle}}
\begin{document}
\title{Viscosity Solutions in Martinet Spaces}

\author{Thomas Bieske, Frederic Bowen}

\address{Thomas Bieske \newline
Department of Mathematics and Statistics,
University of South Florida,
Tampa, FL  33620-5700, USA}
\email{tbieske@usf.edu}

\address{Frederic Bowen \newline
College of Arts and Sciences,
University of South Florida,
Tampa, FL  33620-5700, USA}
\email{bif@usf.edu}


\begin{abstract}
In this paper, we establish the properties of viscosity solutions in Martinet spaces, which lack both the algebraic group law of Carnot groups and the triangular vector fields of Grushin-type spaces. We then prove the uniqueness of viscosity solutions to strictly monotone elliptic PDEs and to the infinite Laplace equation. 
\end{abstract}

\maketitle
\numberwithin{equation}{section}
\newtheorem{theorem}{Theorem}[section]
\newtheorem{corollary}[theorem]{Corollary}
\newtheorem{lemma}[theorem]{Lemma}
\newtheorem{proposition}[theorem]{Proposition}
\newtheorem{remark}[theorem]{Remark}
\newtheorem*{quest}{Main Question}
\newtheorem*{obs}{Observation}
\newtheorem{definition}{Definition}
\allowdisplaybreaks
\section{introduction}
Properties of viscosity solutions in sub-Riemannian spaces have been a topic of recent study. In particular, the existence-uniqueness of viscosity solutions to the $p$-Laplace equation for $1<p\leq \infty$ in Carnot groups was established by exploiting the (non-abelian) algebraic group law \cite{B:MP} and the triangular structure of the vector fields in Grushin-type spaces enabled existence-uniqueness of viscosity solutions to the infinite Laplace and $\infty(x)$-Laplace equations to be proved. \cite{B:GS, B:IX}.

In \cite{CGPV}, existence of viscosity solutions to the infinite Laplace equation in general sub-Riemannian spaces was proved, leading to the question of uniqueness especially in spaces that lack an algebraic group law or lack the vector field structure of Grushin-type spaces. In this paper, we establish uniqueness of viscosity solutions to the infinite Laplace equation in such a space, namely, Martinet spaces.  

This paper is part of an undergraduate research project of the second author under the guidance of the first author. The second author wishes to thank the University of South Florida for the opportunities and support.  

\section{Martinet Space}
We begin with $\mathbb{R}^3$ and coordinates $(x_1,x_2,x_3)$. Consider the vector fields

\begin{eqnarray*}
X_1& = & \frac{\partial}{\partial x_1}\\
\textmd{and\ \ } X_2 & = & \frac{\partial}{\partial x_2}+f(x_1)\frac{\partial}{\partial x_3}
\end{eqnarray*}
where $f:\mathbb{R}\to\mathbb{R}$ is real-analytic. That is, it has a nontrivial Taylor series over $\mathbb{R}$.
The space $\mathbb{R}^3$ with these vector fields, denoted $\mathfrak{m}$, is  the tangent space to the Martinet space $\mathbb{M}$ with coordinates $(x_1,x_2,x_3)$.
Now, $\mathbb{M}$ need not be a group because

\[[X_1,X_2]=f'(x_1)\frac{\partial}{\partial x_3}=\vec{0}\]
when $f'(x_1)=0$ and 

\[[X_1,X_2]=f'(x_1)\frac{\partial}{\partial x_3}\neq\vec{0}\]
when $f'(x_1)\neq0$, meaning the Lie algebra grading can depend on the point $(x_1,x_2,x_3)$, such as the case when $f'(x_1)\not\equiv 0$ but $f'(x_1)$ has zeros. Note that when $f(x_1)=x_1$, we have a presentation of the Heisenberg algebra. We shall denote the vector $[X_1,X_2]$ by $X_3$. 

Because we will assume the zeros of $f(x_1)$ have finite order, there is an $r_{x_1}\in\mathbb{N}$ such that the Lie bracket iteration of minimal length $r_{x_1}$ satisfies the condition $[X_1,[X_1,[...,[X_1,X_2]...]]\neq0$. Then by \cite[Theorem 7.34]{BR:SRG},

\begin{equation}\label{distest}
d(p,q)\approx|x-x_0|+|y-y_0|+|z-z_0|^{(r_{x_1}+1)^{-1}}.
\end{equation}

Even though $\mathbb{M}$ need not be a group, it is a metric space whose natural metric is the Carnot-Carath\'{e}odory distance, which is defined for the points $p$ and $q$ as follows:

\[d(p,q)=\inf_\Gamma\int_0^1\|\gamma'(t)\|dt,\]
where $\|\cdot\|$ is a norm that makes $\{X_1,X_2\}$ an orthonormal set and where $\Gamma$ is the set of all curves $\gamma$ joining p and q with $\gamma(0)=p$, 
$\gamma(1)=q$, and $\gamma'(t)\in\spn\{X_1,X_2\}$.\\
By the Chow-Rashevskii theorem (see e.g.\cite{BR:SRG}), any two points can be connected by such a curve, which means $d(p,q)$ is an honest metric. Using this metric, we can define a Carnot-Carath\'{e}odory ball of radius $r$ centered at a point $p_0$ by

\[B=B(p_0,r)=\{p\in\mathbb{M}:d(p,p_0)<r\};\]
similarly, we shall denote a bounded domain in $\mathbb{M}$ by $\Omega$.
\section{Calculus}
Given a smooth function $u$ on $\mathbb{M}$, we define the horizontal gradient of $u$, denoted $\nabla_0u$ by 
\[\nabla_0u(p)=(X_1u(p),X_2u(p)),\]
the semi-horizontal gradient of $u$, denoted $\nabla_1u$ by 
\[\nabla_1u(p)=(X_1u(p),X_2u(p), X_3u(p)),\]
and the symmetrized second-order (horizontal) derivative matrix, denoted $(D^2u(p))^\star$ by
\[((D^2u(p))^\star)_{ij}=\frac{1}{2}(X_iX_ju(p)+X_jX_iu(p))\]
for $i,j=1,2$.

\begin{definition}
The function $u:\mathbb{M}\to\mathbb{R}$ is said to be $C_{sub}^1$ if $X_1u$ and $X_2u$ are continuous. Similarly, the function $u$ is $C_{sub}^2$ if $X_3u(p)$ is continuous and $X_iX_ju(p)$ is continuous for all $i,j=1,2$.
\end{definition}

It should also be noted that, for any open set $\mathcal{O}\subset\mathbb{M}$, the function $u$ is in the horizontal Sobolev space $W^{1,q}(\mathcal{O})$ if $u, X_1u,X_2u$ are in $L^q(\mathcal{O})$. Replacing $L^q(\mathcal{O})$ by $L^q_{loc}(\mathcal{O})$, the space $W^{1,q}_{loc}(\mathcal{O})$ is defined similarly. The space $W^{1,q}_0(\mathcal{O})$ is the closure in $W^{1,q}(\mathcal{O})$ of smooth functions with compact support. Locally Lipschitz functions are those functions $u$ such that 

\[\|\nabla_0u\|_{L^\infty_{loc}}<\infty.\]
\section{Martinet Jets and Viscosity Solutions}
Using the derivatives of the previous section, the class of equations we consider are given by

\begin{equation}\label{geneq}
F(p,u(p),\nabla_1u(p),(D^2u(p))^\star)=0
\end{equation}
where the continuous function 

\[F:\mathbb{M}\times\mathbb{R}\times  \mathfrak{m}\times S^2\to\mathbb{R}\]
satisfies

\[F(p,r,\eta,X)\leq F(p,s,\eta, Y).\]
when $r\leq s$ and $Y\leq X$. (That is, $F$ is proper \cite{CIL:UGTVS}). Recall that $S^2$ is the set of $2\times 2$ real symmetric matrices. An example of this type of equation is the quasilinear horizontal q-Laplacian

\[\div(\|\nabla_0u\|^{q-2}\nabla_0u)=X_1(\|\nabla_0u\|^{q-2}X_1u)+X_2(\|\nabla_0u\|^{q-2}X_2u)\] 
for $2\leq q<\infty$. Formally taking the limit as $q\to\infty$ yields the horizontal infinite Laplacian

\[\Delta_{\infty}f=\sum_{i,j=1}^2X_ifX_jfX_iX_jf=\langle(D^2f)^\star\nabla_0f,\nabla_0f\rangle.\]

We now wish to define solutions to the equation

\[F(p,u(p),\nabla_1u(p),(D^2u(p))^\star)=0\]
in the viscosity sense. In order to do so, we must define the subelliptic jets. (For a thorough discussion of jets, the interested reader is directed to \cite{CIL:UGTVS}.) Given an open set $\mathcal{O}\subset\mathbb{M}$, a point $p_0\in\mathcal{O}$, and a function $u: \mathcal{O}\to\mathbb{R}$, we consider classes of test functions which ``touch” the function $u$ at the point $p_0$. Namely, we have the so-called ``touching above”
functions
\[\mathcal{TA}(u,p_0)=\{\varphi\in C_{sub}^2(\Omega):0 = \varphi(p_0)-u(p_0) <\varphi(p)-u(p)\:\text{near}\: p_0\};\]
we have also the ``touching below” functions at $p_0$ defined by
\[\mathcal{TB}(u,p_0)=\{\varphi\in C_{sub}^2(\Omega): 0 = u(p_0)-\psi(p_0) <u(p)-\psi(p)\:\text{near}\: p_0\}.\]
We then use these test functions to define Martinet jets.
\begin{definition}
Given $u:\Omega\subset \mathbb{M}\to\mathbb{R}$, we define the second-order superjet for $u$ by
\[J^{2,+} u(p_0)=\{(\nabla_1\varphi(p_0),(D^2\varphi)^\star(p_0))\in\mathfrak{m}\times S^2:\varphi\in\mathcal{TA}(u,p_0)\}\]
and the second-order subjet for $u$ by 
\[J^{2,-} u(p_0)=\{(\nabla_1\psi(p_0),(D^2\psi)^\star(p_0))\in\mathfrak{m}\times S^2:\psi\in\mathcal{TB}(u,p_0)\}\]
\end{definition}
It is easy to see from the definition that $J^{2,-}u(p_0)=-J^{2,+}(-u)(p_0)$. In addition, we say that the ordered pair
$(\eta,X)\in \mathfrak{m}\times S^2$ belongs to the closure of the superjet, written $(\eta,X)\in \overline{J}^{2,+} u(p_0)$, if there exists a sequence of points 
$\{p_k\}\subset\Omega$ and jet entries $(\eta_k,X_k)\in J^{2,+}u(p_k)$ so that
\[(p_k,u(p_k),\eta_k,X_k)\to(p_0,u(p_0),\eta,X).\]

\begin{definition}
Let $\mathcal{O}$ be a open set in $\mathbb{M}$ and let $u:\mathcal{O}\to\mathbb{R}$. If $u$ is upper semicontinuous and 
\[F(p,u(p),\eta,X)\leq0\quad\text{for all}\quad p\in\mathcal{O},\quad\text{for all}\quad (\eta,X)\in \overline{J}^{2,+}_{\mathcal{O}}u(p)\]
then u is a viscosity subsolution of $F(p,u(p),\nabla_1u(p),(D^2u(p))^\star)=0.$
If $u$ is lower semicontinuous and 
\[F(p,u(p),\eta,X)\geq0\quad\text{for all}\quad p\in\mathcal{O},\quad\text{for all}\quad (\eta,X)\in \overline{J}^{2,-}_{\mathcal{O}}u(p)\]
then u is a viscosity supersolution of $F(p,u(p),\nabla_1u(p),(D^2u(p))^\star)=0.$\\
The function $u$ is a viscosity solution if it is both a viscosity subsolution and a viscosity supersolution. 
\end{definition}

In order to use the machinery of \cite{CIL:UGTVS} to prove comparison principles, a relationship between Euclidean and Martinet jets must be established. This is accomplished through the following lemma, whose proof is a special case of \cite[Corollary 3.2]{B:MP} and omitted. 

\begin{lemma}[Twisting Lemma]\label{twlemma}
Let $O\subset \mathbb{M}$ be open, let $u: O\to\mathbb{R}$, and let $p_0=(x_1^0,x_2^0,x_3^0) \in O$. Suppose that $(\eta, X)\in J^{2,+}_{\textmd{eucl}} u(p_0)$ with $\eta=(\eta_1,\eta_2,\eta_3)$ and $X=\{X_{ij}\}$ a $3\times 3$ symmetric matrix. Then 
\begin{equation}\label{twisty}
(B(p_0)\cdot \eta, A(p_0)\cdot X\cdot A^T(p_0)+\mathcal{T}(\eta,p_0))\in J^{2,+} u(p_0)
\end{equation}
where $A(p_0)$ is the $2\times3$ matrix of Martinet vector coefficients defined by
\[A(p_0)=\begin{pmatrix}1&0&0\\0&1&f(x_1^0)\end{pmatrix}_,\]
the $3\times 3$ matrix $B(p_0)$ is given by 
\[B(p_0)=A(p_0)\oplus(0,0,f'(x_1^0)=\begin{pmatrix}1&0&0\\0&1&f(x_1^0)\\0&0&f'(x_1^0) \end{pmatrix}_,\]
and the $2\times2$ twisting matrix $\mathcal{T}(\eta,p_0)$ is given by
\[\mathcal{T}(\eta,p_0)=\frac{\eta_3}{2}\begin{pmatrix}0&f'(x_1^0)\\
f'(x_1^0)&0\end{pmatrix}_.\]
More specifically,
\[B(p_0)\cdot\eta=\begin{pmatrix}
\eta_1\\
\eta_2+f(x_1^0)\eta_3\\
f'(x_1^0)\eta_3
\end{pmatrix}\]
and 
\begin{eqnarray*}
\lefteqn{A(p_0)\cdot X\cdot A^T(p_0)+\mathcal{T}(\eta,p_0)=} & & \\
& & \begin{pmatrix}x_{11}&x_{12}+f(x_1^0)x_{13}+\frac{1}{2}f'(x_1^0)\eta_3\\
\mbox{} & \mbox{} \\
x_{12}+f(x_1^0)x_{13}+\frac{1}{2}f'(x_1^0)\eta_3&x_{22}+2f(x_1^0)x_{23}+f(x_1^0)^2x_{33}\end{pmatrix}_.
\end{eqnarray*}
\end{lemma}

\section{Maximum Principle}
We begin by stating a lemma analogous to \cite[Lemma 3.1]{CIL:UGTVS}. The proof is similar and thus is omitted. 

\begin{remark}
We wish to emphasize that in Lemma \ref{mmp} below,  the penalty function $\varphi(p,q)$ not only differs from the Euclidean function in \cite[Lemma 3.1]{CIL:UGTVS}, but also differs from the Grushin-type spaces penalty function of \cite[Lemma 4.1]{B:GS}. It will naturally differ from the penalty function of the Carnot Group Maximum Principle \cite[Lemma 3.6]{B:MP} because that penalty function incorporates the algebraic group law of the Carnot group.

\end{remark}
\begin{lemma}\label{mmp}
Let $u$ be an upper semicontinuous function in $\Omega$ and $v$ be a lower semicontinuous function in $\Omega$. For $\tau>0$ and for points $p$ and $q$ given by $p=(x_1,x_2,x_3)$ and $q=(y_1,y_2,y_3)$, let the function $\varphi(p,q)$ be defined by
\[\varphi(p,q)\equiv\frac{1}{2}(x_1-y_1)^2+\frac{1}{4}(x_2-y_2)^4+\frac{1}{4}(x_3-y_3)^4\]
and let the function $M_\tau$ be defined by
\[M_\tau=\sup_{\overline{\Omega}\times\overline{\Omega}}(u(p)-v(q)-\tau\varphi(p,q)).\]
Let $p_\tau=(x_1^\tau,x_2^\tau,x_3^\tau)$ and $q_\tau=(y_1^\tau,y_2^\tau,y_3^\tau)$ be such that 
\[\lim_{\tau\to\infty}(M_\tau-(u(p_\tau)-v(q_\tau)-\tau\varphi(p_\tau,q_\tau)))=0\]
Then, 
\begin{equation}\lim_{\tau\to\infty}\tau\varphi(p_\tau,q_\tau)=0\end{equation}
and
\begin{equation}\lim_{\tau\to\infty}M_\tau=u(p^*)-v(p^*)=\sup_{\overline\Omega}(u(p)-v(p))\end{equation}
whenever $p^*$ is a limit point of $p_\tau$ as $\tau\to\infty$.
\end{lemma}
By the Theorem of Sums \cite[Theorem 3.2]{CIL:UGTVS}  there are matrices $X,Y\in S^3$ so that 
\begin{eqnarray*}
(\tau \nabla_p\varphi(p_\tau,q_\tau), X)  \in  \overline{J}^{2,+}_{\eucl}u(p_\tau) \ \ 
\textmd{and}\ \ (-\tau\nabla_q\phi(p_\tau,q_\tau), Y) \in  \overline{J}^{2,-}_{\eucl}v(q_\tau)
\end{eqnarray*}
where $\nabla_p$ and $\nabla_q$ denote the Euclidean gradient with respect to the point $p$ and $q$, respectively. Invoking the Twisting Lemma (Lemma \ref{twlemma}), there are vectors $\Upsilon_{p_\tau}, \Upsilon_{q_\tau}\in \mathfrak{m}$ and matrices $\mathcal{X}_\tau, \mathcal{Y}_\tau \in S^2$ so that  
\begin{eqnarray*}
(\tau\Upsilon_{p_\tau},\mathcal{X}_\tau) \in \overline{J}^{2,+}u(p_\tau) \  \textmd{and}\ \ 
(\tau\Upsilon_{q_\tau},\mathcal{Y}_\tau) \in \overline{J}^{2,-}v(q_\tau).
\end{eqnarray*}

The properties of these jet elements are given in the following lemma. 
\begin{lemma}\label{5.2}
The vectors $\Upsilon_{p_\tau}$ and $\Upsilon_{q_\tau}$ satisfy
\begin{equation}\|\Upsilon_{p_\tau}\|^2-\|\Upsilon_{q_\tau}\|^2=O(\varphi(p_\tau,q_\tau)^2).\end{equation}
In addition, with the usual ordering, the matrix $\mathcal{X}^\tau$ is smaller than the matrix $\mathcal{Y}^\tau$ with an error term. In particular we have $\mathcal{X}^\tau\leq\mathcal{Y}^\tau+\mathcal{R}^\tau$, where $\mathcal{R}^\tau\rightarrow0$ as $\tau\rightarrow\infty$
\end{lemma}

\begin{proof}
A straightforward computation shows 
\begin{eqnarray*}
\|\Upsilon_{p_\tau}\|^2-\|\Upsilon_{q_\tau}\|^2 & = & 2(x_2-y_2)^3(x_3-y_3)^3(f(x_1)-f(y_1)) \\ 
& & \mbox{}+(x_3-y_3)^6(f(x_1)^2-f(y_1)^2)+(f'(x_1)-f'(y_1))(x_3-y_3)^6
\end{eqnarray*}
Now, \begin{equation*}
(x_2-y_2) \sim  \varphi^\frac{1}{4},\
(x_3-y_3)  \sim  \varphi^\frac{1}{4},\ \textmd{and\ }
f(x_1)-f(y_1)  \sim  (x_1-y_1)\sim\varphi^\frac{1}{2}.
\end{equation*}
We then have 
\[\|\Upsilon_{p_\tau}\|^2-\|\Upsilon_{q_\tau}\|^2\sim (\varphi^\frac{8}{4}+\varphi^\frac{8}{4}+\varphi^\frac{8}{4}).\]
The vector result follows. 

Next, we focus on the matrix difference estimate.  Using the Twisting Lemma (Lemma \ref{twlemma}), we have (writing $A_p$ for $A(p_\tau)$, $\mathcal{T}_p$ for $\mathcal{T}(\Upsilon_{p_\tau},p_\tau))$, etc),

\begin{eqnarray*}
\langle \mathcal{X}^\tau\varepsilon,\varepsilon\rangle-\langle \mathcal{Y}^\tau\kappa,\kappa\rangle
&=& \langle(A_pX^\tau A^T_p+\mathcal{T}_p)\epsilon,\epsilon\rangle-\langle(A_qY^\tau A^T_q+\mathcal{T}_q)\kappa,\kappa\rangle\\
&=&\langle X^\tau A^T_p\epsilon,A^T_p\epsilon\rangle-\langle Y^\tau A^T_q\kappa,A^T_q\kappa\rangle+\langle \mathcal{T}_p\epsilon,\epsilon\rangle-\langle \mathcal{T}_q\kappa,\kappa\rangle\\
& &\leq\tau\left\langle C
\begin{pmatrix} A^T_p\epsilon\\A^T_q\kappa \end{pmatrix},
\begin{pmatrix} A^T_p\epsilon\\A^T_q\kappa \end{pmatrix}\right\rangle +\langle \mathcal{T}_p\epsilon,\epsilon\rangle-\langle \mathcal{T}_q\kappa,\kappa\rangle \\
 & \stackrel{\textmd{def}}{=} & \tau\mathcal{V}+\mathcal{W}. 
\end{eqnarray*}
Here the matrix $C$ is given by 
\begin{eqnarray*}
\lefteqn{C=(D^2\varphi)_{\eucl}=}& & \\
& \begin{pmatrix}
1& 0 & 0 &-1 & 0 & 0 \\
0& 3(x_2-y_2)^2 & 0 & 0 & -3(x_2-y_2)^2 & 0 \\
0& 0 & 3(x_3-y_3)^2 & 0 & 0 & 3(x_3-y_3)^2 \\
-1& 0 & 0 & 1 & 0 & 0 \\
0 & -3(x_2-y_2)^2 & 0 & 0& 3(x_2-y_2)^2 & 0 \\
 0 & 0 & 3(x_3-y_3)^2 & 0& 0 & 3(x_3-y_3)^2 
\end{pmatrix}_. 
\end{eqnarray*}
Let us first simplify the inner product given by the term $\tau\mathcal{V}$. 
Unravelling the matrix multiplication, we have 
\begin{equation*}
C \begin{pmatrix} A^T_p\epsilon\\A^T_q\kappa \end{pmatrix} = 
\begin{pmatrix}
\epsilon_1-\kappa_1\\
 3(x_2^\tau-y_2^\tau)^2 (\epsilon_2-\kappa_2)\\
 3(x_3^\tau-y_3^\tau)^2(f(x_1^\tau)\epsilon_2-f(y_1^\tau)\kappa_2)\\
 -\epsilon_1+\kappa_1\\
 3(x_2^\tau-y_2^\tau)^2(-\epsilon_2+\kappa_2)\\ 
 3(x_3^\tau-y_3^\tau)^2(-f(x_1^\tau)\epsilon_2+f(y_1^\tau)\kappa_2)
 \end{pmatrix}
\end{equation*}
so that $\tau \mathcal{V}$ is given by 
\begin{equation*}
\tau\mathcal{V} = \tau\bigg((\epsilon_1-\kappa_1)^2+ 3(x_2^\tau-y_2^\tau)^2(\epsilon_2-\kappa_2)^2+  3(x_3^\tau-y_3^\tau)^2(f(x_1^\tau)\epsilon_2-f(y_1^\tau)\kappa_2)^2\bigg)
\end{equation*}
We note that when $\epsilon=\kappa$, we have 
\begin{eqnarray}\label{Vlabel}
\tau \mathcal{V} & \sim & \tau\bigg((x_3^\tau-y_3^\tau)^2(f(x_1^\tau)-f(y_1^\tau))^2\bigg) \sim \tau \varphi^{\frac{1}{2}} \times \varphi
\end{eqnarray}
We now focus on the $\mathcal{W}$ term. We have 
\begin{equation*}
\mathcal{W}=\langle \mathcal{T}_p\epsilon,\epsilon\rangle-\langle \mathcal{T}_q\kappa,\kappa\rangle =
\tau (x_3^\tau-y_3^\tau)^3\bigg(f'(x_1^\tau)\epsilon_1\epsilon_2-f'(y_1^\tau)\kappa_1\kappa_2\bigg).
\end{equation*}
When $\epsilon=\kappa$, we then have 
\begin{equation}\label{Wlabel}
\mathcal{W}\sim \tau\varphi^{\frac{3}{4}}\varphi^{\frac{1}{2}}
\end{equation}
The matrix relation then follows from Equations \eqref{Vlabel} and \eqref{Wlabel}. 
\end{proof}

This lemma is the key to proving comparison principles. The first comparison principle involves strictly monotone elliptic equations. Such equations satisfy the following properties:

\begin{eqnarray*}
\sigma(r-s)&\leq & F(p,r,\eta,X)-F(p,s,\eta,X),\\
|F(p,r,\eta,X)-F(q,r,\eta,X)|&\leq & w_1(d_C(p,q)),\\
|F(p,r,\eta,X)-F(p,r,\eta,Y)|&\leq & w_2(\|Y-X\|),\\
|F(p,r,\eta,X)-F(p,r,\nu,X)|&\leq & w_3(\|\eta\|-\|\nu\|),
\end{eqnarray*}
where the constant $\sigma>0$ and the functions $w_i:[0,\infty]\to[0,\infty]$ satisfy $w_i(0^+)=0$ for $i=1,2,3$. The appropriate comparison principle is given next.

\begin{theorem}
Let $F$ satisfy the four properties just listed. Let $u$ be an upper semicontinuous subsolution and $v$ a lower semicontinuous supersolution to 
\[F(p,g(p),\nabla_1g(p), (D^2g(p))^\star)=0\]
in a domain $\Omega$ such that
\[\limsup_{q\rightarrow p}u(q)\leq\liminf_{q\rightarrow p}v(q)\]
when $p\in\partial\Omega$, where both sides are neither $\infty$ nor $-\infty$ simultaneously. Then
\[u(p)\leq v(p)\]
for all $p\in\Omega$.
\end{theorem}
\begin{proof}
Suppose $\sup_\Omega (u-v)>0.$ Using the Martinet maximum principle (Lemma \ref{mmp}), we obtain
\begin{align*}
\sigma (u(p_\tau)-v(q_\tau))&\leq F(p_\tau,u(p_\tau),\tau\Upsilon_{p_\tau}, \mathcal{X}^\tau)-F(p_\tau,v(q_\tau),\tau\Upsilon_{p_\tau}, \mathcal{X}^\tau)\\
&=F(p_\tau,u(p_\tau),\tau\Upsilon_{p_\tau}, \mathcal{X}^\tau)-F(q_\tau,v(q_\tau),\tau\Upsilon_{q_\tau}, \mathcal{Y}^\tau)\\
&+F(q_\tau,v(q_\tau),\tau\Upsilon_{q_\tau}, \mathcal{Y}^\tau)-F(p_\tau,v(q_\tau),\tau\Upsilon_{q_\tau}, \mathcal{Y}^\tau)\\
&+F(p_\tau,v(q_\tau),\tau\Upsilon_{q_\tau}, \mathcal{Y}^\tau)-F(p_\tau,v(q_\tau),\tau\Upsilon_{p_\tau}, \mathcal{Y}^\tau)\\
&+F(p_\tau,v(q_\tau),\tau\Upsilon_{p_\tau}, \mathcal{Y}^\tau)-F(p_\tau,v(q_\tau),\tau\Upsilon_{p_\tau}, \mathcal{X}^\tau).
\end{align*}
The first term is negative because $u$ is a subsolution and $v$ is a supersolution. Using (Lemma \ref{5.2}) yields
\[0<\sigma (u(p_\tau)-v(q_\tau))\leq w_1(d_C(p_\tau,q_\tau))+w_2(||R_\tau||)+w_3(\tau\big|||\Upsilon_{q_\tau}||-||\Upsilon_{p_\tau}||\big|),\]
which goes to 0 as $\tau$ approaches $\infty $.
\end{proof}
\section{Viscosity Solutions to the Infinite Laplace Equation}
In this section, we look to address the existence-uniqueness of viscosity solutions to the $\infty$-Laplace equation, also called as infinite harmonic functions. By \cite{CGPV}, infinite harmonic functions exist, so we will focus on uniqueness. Unfortunately, the delicate nature of the estimates used in the uniqueness proof do not allow us to use the standard machinery, even when considering the results of the previous section. However, we do have a helpful tool at our disposal, namely, the Iterated Maximum Principle \cite{B:GS}. 
\subsection{The Iterated Maximum Principle}
We will need a specific version of the Iterated Maximum Principle to match our Martinet environment: 
\begin{lemma}[The Iterated Maximum Principle]\label{IMP}\cite{B:GS}
Let $\Omega\subset\mathbb{M}$ be a domain and let  $u$ be an upper semicontinuous function in $\Omega$ and $v$ a lower semicontinuous function in $\Omega$.  Assume that there exists some $p_0\in\Omega$ so that
\[u(p_0)-v(p_0)>0.\]
Let $\vec{\tau}=(\tau_1,\tau_2,\tau_3)\in\mathbb{R}^3$ have positive coordinates and, for each pair of points in $\mathbb{M}$ $p=(x_1,x_2,x_3),q=(y_1,y_2,y_3)$ define the functions 
\begin{align*}
\varphi_{\tau_1,\tau_2,\tau_3}(p,q)&:=\frac{1}{2}\sum_{k=1}^3\tau_k(x_k-y_k)^2\\
\varphi_{\tau_2,\tau_3}(p,q)&:=\frac{1}{2}\sum_{k=2}^3\tau_k(x_k-y_k)^2\\
\varphi_{\tau_3}(p,q)&:=\frac{1}{2}\tau_3(x_3-y_3)^2.
\end{align*}
Appealing to the compactness of $\overline{\Omega}$ and to the upper semicontinuity, we may also define 
\begin{align*} 
M_{\tau_1,\tau_2,\tau_3}&:=\sup_{\overline{\Omega}\times\overline{\Omega}}\{u(p)-v(q)-\varphi_{\tau_1,\tau_2,\tau_3}(p,q)\}\\
&=u(p_{\tau_1,\tau_2,\tau_3})-v(q_{\tau_1,\tau_2,\tau_3})-\varphi_{\tau_1,\tau_2,\tau_3}(p_{\tau_1,\tau_2,\tau_3},q_{\tau_1,\tau_2,\tau_3})\\
M_{\tau_2,\tau_3}&:=\sup_{\overline{\Omega}\times\overline{\Omega}}\{u(p)-v(q)-\varphi_{\tau_2,\tau_3}(p,q):x_1=y_1\}\\
&=u(p_{\tau_2,\tau_3})-v(q_{\tau_2,\tau_3})-\varphi_{\tau_2,\tau_3}(p_{\tau_2,\tau_3},q_{\tau_2,\tau_3})\\
M_{\tau_3}&:=\sup_{\overline{\Omega}\times\overline{\Omega}}\{u(p)-v(q)-\varphi_{\tau_3}(p,q):x_2=y_2\}\\
&=u(p_{\tau_3})-v(q_{\tau_3})-\varphi_{\tau_3}(p_{\tau_3},q_{\tau_3}).
\end{align*}
Then,
\[\lim_{\tau_3\to\infty}\lim_{\tau_2\to\infty}\lim_{\tau_1\to\infty}M_{\tau_1,\tau_2,\tau_3}=u(p_0)-v(p_0)\]
and
\[\lim_{\tau_3\to\infty}\lim_{\tau_2\to\infty}\lim_{\tau_1\to\infty}\varphi_{\tau_1,\tau_2,\tau_3}(p_{\tau_1,\tau_2,\tau_3},q_{\tau_1,\tau_2,\tau_3})=0.\]
Additionally, the first $l$ coordinates of $p_{\tau_{l+1},...,\tau_3}$ and $q_{\tau_{l+1},...,\tau_3}$ are identical. That is, 
\[x_k^{\tau_{l+1},...,\tau_3}=y_k^{\tau_{l+1},...,\tau_3}\:,\:k=1,...l.\]
\end{lemma}
The proof of the iterated Maximum Principle leads immediately to the following results which permit us to take the parameters $\tau_k\to\infty$ in any order, and to speak of the full limit as $\tau_{k_1},\tau_{k_2},\tau_{k_3}\to\infty$.
\begin{corollary}\cite[Corollary 4.4]{B:GS} Under the conditions of (Lemma \ref{IMP}), each iterated limit of $M_{\tau_{k_1},\tau_{k_2},\tau_{k_3}}=u(p_0)-v(p_0)$. In other words, 
\[\lim_{\tau_{k_1}\to\infty}\lim_{\tau_{k_2}\to\infty}\lim_{\tau_{k_3}\to\infty}M_{\tau_{k_1},\tau_{k_2},\tau_{k_3}}(p_{\tau_{k_1},\tau_{k_2},\tau_{k_3}},q_{\tau_{k_1},\tau_{k_2},\tau_{k_3}})=u(p_0)-v(q_0).\]
Consequently,
\[\lim_{\tau_{k_1}\to\infty}\lim_{\tau_{k_2}\to\infty}\lim_{\tau_{k_3}\to\infty}\varphi_{\tau_{k_1},\tau_{k_2},\tau_{k_3}}(p_{\tau_{k_1},\tau_{k_2},\tau_{k_3}},q_{\tau_{k_1},\tau_{k_2},\tau_{k_3}})=0.\]
\end{corollary}
\begin{lemma}\cite[Lemma 4.5]{B:GS}
Under the conditions of (Lemma \ref{IMP}), the full limit of $M_{\tau_1,\tau_2,\tau_3}$ exists and is equal to $u(p_0)-v(p_0)$. More precisely,
\[\lim_{\tau_{1},\tau_{2},\tau_{3}\to\infty}M_{\tau_{1},\tau_{2},\tau_{3}}=u(p_0)-v(p_0).\]
In addition,
\[\lim_{\tau_1,\tau_2,\tau_3\to\infty}\varphi_{\tau_1,\tau_2,\tau_3}(p_{\tau_1,\tau_2,\tau_3},q_{\tau_1,\tau_2,\tau_3})=0.\]
\end{lemma}
\begin{remark}
Owing to the above lemma, there is no ambiguity in relabeling the intermediate points $p_{\tau_1,\tau_2,\tau_3},q_{\tau_1,\tau_2,\tau_3}$ and function $\varphi_{\tau_1,\tau_2,\tau_3}$ as $p_{\vec{\tau}},q_{\vec{\tau}},\text{and}\;\varphi_{\vec{\tau}}.$ We will also denote the coordinates of $p_{\vec{\tau}},q_{\vec{\tau}}$ as $x_k^{\vec{\tau}},y_k^{\vec{\tau}}$ respectively.
\end{remark}
\subsection{Infinite Harmonic Functions}
\begin{lemma}\label{inflemma}
Let $u,v,\varphi_{\vec{\tau}}$ and $(p_{\vec{\tau}},q_{\vec{\tau}})$ be as in the Iterated Maximum Principle (Lemma \ref{IMP}). Assume that at least one of the functions $u,v$ is locally $\mathbb{M}$-Lipschitz. Then:
\begin{enumerate}
\renewcommand{\labelenumi}{(\Alph{enumi})}
\item There exist $(\Upsilon^+_{\vec{\tau}},\mathcal{X}_{\vec{\tau}})\in\overline{J}^{2,+}u(p_{\vec{\tau}})$ and $(\Upsilon^-_{\vec{\tau}},\mathcal{Y}_{\vec{\tau}})\in\overline{J}^{2,-}v(q_{\vec{\tau}}).$
\item Define $(p\diamond q)_k$ to be the point whose k-th coordinate coincides with q and whose other coordinates coincide with p. In other words,
\begin{equation*}
(p\diamond q)_1 = (y_1,x_2,x_3), 
(p\diamond q)_2 = (x_1,y_2,x_3), 
\textmd{and\ }(p\diamond q)_3 = (x_1,x_2,y_3).
\end{equation*}
Then for each index k,
\begin{equation}\label{7.1}
\tau_k(x_k^{\vec{\tau}}-y_k^{\vec{\tau}})^2\lesssim d_{CC}(p_{\vec{\tau}},(p_{\vec{\tau}}\diamond q_{\vec{\tau}})_k).
\end{equation}
In particular, for $k=1,2$ we have
\begin{equation}\label{7.2}
\tau_k|x_k^{\vec{\tau}}-y_k^{\vec{\tau}}|=O(1)\;\text{as}\;\tau_k\rightarrow \infty.
\end{equation}
\item Let $\eta^+_{\vec{\tau}}$ and $\eta^-_{\vec{\tau}}$ be the horizontal parts of $\Upsilon^+_{\vec{\tau}}$ and $\Upsilon^-_{\vec{\tau}}$ respectively. That is, from the Twisting Lemma (Lemma \ref{twlemma}), 
\[\left\{
\begin{aligned}
\eta^+_{\vec{\tau}}&= A(p_{\vec{\tau}})\cdot D_{\eucl}(p)\varphi_{\vec{\tau}}(p_{\vec{\tau}},q_{\vec{\tau}})\\
\eta^-_{\vec{\tau}}&= A(q_{\vec{\tau}})\cdot -D_{\eucl}(q)\varphi_{\vec{\tau}}(p_{\vec{\tau}},q_{\vec{\tau}}).
\end{aligned}
\right.\]
We then have that the vector estimate 
\begin{equation}\label{7.3}
\lim_{\tau_3\to\infty}\lim_{\tau_2\to\infty}\lim_{\tau_1\to\infty}\big|||\eta^+_{\vec{\tau}}||^2-||\eta^-_{\vec{\tau}}||^2\big|=0
\end{equation}
holds.
\item Let $\eta^+_{\vec{\tau}}$ and $\eta^-_{\vec{\tau}}$ be as above. Then, the matrix estimate 
\begin{equation}\label{7.4}
\lim_{\tau_3\to\infty}\lim_{\tau_2\to\infty}\lim_{\tau_1\to\infty}\langle\mathcal{X}_{\vec{\tau}}\cdot\eta^+_{\vec{\tau}},\eta^+_{\vec{\tau}}\rangle-\langle\mathcal{Y}_{\vec{\tau}}\cdot\eta^-_{\vec{\tau}},\eta^-_{\vec{\tau}}\rangle=0.
\end{equation}
holds.
\end{enumerate}
\end{lemma}
\begin{proof}
\textbf{Item (A)}: This follows from the Theorem of Sums \cite[Theorem 3.2]{CIL:UGTVS} and the Martinet Twisting Lemma (Lemma \ref{twlemma}).\\
\textbf{Item (B)}:
By the definition of the points $p_{\vec{\tau}}, q_{\vec{\tau}}$ for all points $p,q\in\Omega$ the inequality
\[u(p)-v(q)-\varphi_{\vec{\tau}}(p,q)\leq u(p_{\vec{\tau}})-v(q_{\vec{\tau}})-\varphi_{\vec{\tau}}(p_{\vec{\tau}},q_{\vec{\tau}})\]
is satisfied. Hence assuming (without loss of generality) that $u$ is $\mathbb{M}$-Lipschitz,  setting $p:=(p_{\vec{\tau}}\diamond q_{\vec{\tau}})_k$ and $q:=q_{\vec{\tau}}$, and recollecting terms, we obtain
\begin{align}
\nonumber\tau_k(x_k^{\vec{\tau}}-y_k^{\vec{\tau}})^2&=\varphi_{\vec{\tau}}(p_{\vec{\tau}},q_{\vec{\tau}})-\varphi_{\vec{\tau}}(p_{\vec{\tau}}\diamond q_{\vec{\tau}})_k,q_{\vec{\tau}})\\
&\leq u(p_{\vec{\tau}})-u((p_{\vec{\tau}}\diamond q_{\vec{\tau}})_k) \leq Kd_{CC}(p_{\vec{\tau}},(p_{\vec{\tau}}\diamond q_{\vec{\tau}})_k) \label{7.5}
\end{align}
where $K$ is the Lipschitz constant for $u$. This is Equation \eqref{7.1}.

 To complete Item B, we turn our attention to the expression $\tau_k|x_k^{\vec{\tau}}-y_k^{\vec{\tau}}|$ for $k=1,2$. If $x_k^{\vec{\tau}}\neq y_k^{\vec{\tau}}$ then  Equation \eqref{7.5} leads to
\begin{equation}\label{7.6}
\tau_k|x_k^{\vec{\tau}}-y_k^{\vec{\tau}}|=\tau_k(x_k^{\vec{\tau}}-y_k^{\vec{\tau}})^2\cdot\frac{1}{|x_k^{\vec{\tau}}-y_k^{\vec{\tau}}|}\leq\frac{Kd_{CC}(p_{\vec{\tau}}, (p_{\vec{\tau}}\diamond q_{\vec{\tau}})_k)}{|x_k^{\vec{\tau}}-y_k^{\vec{\tau}}|}.
\end{equation}
Now for $k=1,2$, Equation \eqref{distest} tells us
\begin{equation}\label{7.7}
d_{CC}(p_{\vec{\tau}}, (p_{\vec{\tau}}\diamond q_{\vec{\tau}})_k)\approx |x_k^{\vec{\tau}}-y_k^{\vec{\tau}}|.
\end{equation}
Equation \eqref{7.2} follows.\\ 
\textbf{Item (C)}:
Direct calculation shows
\[\frac{\partial }{\partial x_k}\varphi_{\vec{\tau}}(p_{\vec{\tau}},q_{\vec{\tau}})=\tau_k(x_k^{\vec{\tau}}-y_k^{\vec{\tau}})=-\frac{\partial }{\partial y_k}\varphi_{\vec{\tau}}(p_{\vec{\tau}},q_{\vec{\tau}})\]
so we conclude that
\begin{eqnarray}\label{etadef}
\eta^+_{\vec{\tau}}=
\begin{pmatrix}
\tau_1(x^{\vec{\tau}}_1-y^{\vec{\tau}}_1)\\
\tau_2(x^{\vec{\tau}}_2-y^{\vec{\tau}}_2) +f(x_1^{\vec{\tau}})\tau_3(x^{\vec{\tau}}_3-y^{\vec{\tau}}_3)
\end{pmatrix}
\textmd{\ and\ }
\eta^-_{\vec{\tau}}=
\begin{pmatrix}
\tau_1(x^{\vec{\tau}}_1-y^{\vec{\tau}}_1)\\
\tau_2(x^{\vec{\tau}}_2-y^{\vec{\tau}}_2)+f(y_1^{\vec{\tau}})\tau_3(x^{\vec{\tau}}_3-y^{\vec{\tau}}_3)
\end{pmatrix}
\end{eqnarray}
This leads us to:
\[\big|||\eta_\tau^+||^2-||\eta_\tau^-||^2\big|=2\tau_2\tau_3(x^{\vec{\tau}}_2-y^{\vec{\tau}}_2)(x^{\vec{\tau}}_3-y^{\vec{\tau}}_3)\big(f(x_1^{\vec{\tau}})-f(y_1^{\vec{\tau}})\big)+\tau^2_3(x^{\tau}_3-y^{\vec{\tau}}_3)^2\big(f(x_1^{\vec{\tau}})^2-f(y_1^{\vec{\tau}})^2\big).\]
Now,
\begin{align*}
\lim_{\tau_1\to \infty}\big|||\eta_\tau^+||^2-||\eta_\tau^-||^2\big|&=2\tau_2\tau_3(x^{\vec{\tau}}_2-y^{\vec{\tau}}_2)(x^{\vec{\tau}}_3-y^{\vec{\tau}}_3)\big(f(x_1^0)-f(x_1^0)\big)+\tau_3^2(x^{\vec{\tau}}_3-y^{\vec{\tau}}_3)^2\big(f(x_1^0)^2-f(x_1^0)^2\big)\\
&=2\tau_2\tau_3(x^{\vec{\tau}}_2-y^{\vec{\tau}}_2)(x^{\vec{\tau}}_3-y^{\vec{\tau}}_3)(0)+\tau_3^2(x^{\vec{\tau}}_3-y^{\vec{\tau}}_3)^2(0)\\
&=0
\end{align*}
and thus,
\[\lim_{\tau_3\to \infty}\lim_{\tau_2\to \infty}\lim_{\tau_1\to \infty}\big|||\eta_\tau^+||^2-||\eta_\tau^-||^2\big|=0.\]
The result then follows. \\
\textbf{Item (D)}
From the Twisting Lemma (Lemma \ref{twlemma}) we have 
\begin{equation*}
\langle \mathcal{X}^{\vec{\tau}}\eta^+_{\vec{\tau}}, \eta^+_{\vec{\tau}}\rangle - 
\langle \mathcal{Y}^{\vec{\tau}}\eta^-_{\vec{\tau}}, \eta^-_{\vec{\tau}}\rangle = I_1 + I_2,
\end{equation*}
where
\begin{align*}
I_1 &:= \Big\langle \Big(A(p^{\vec{\tau}})\cdot X^{\vec{\tau}} \cdot A^T(p^{\vec{\tau}})\Big)  \eta^+_{\vec{\tau}}, \eta^+_{\vec{\tau}} \Big\rangle
-  \Big\langle \Big(A(q^{\vec{\tau}})\cdot Y^{\vec{\tau}} \cdot A^T(q^{\vec{\tau}})\Big)  \eta^-_{\vec{\tau}}, \eta^-_{\vec{\tau}} \Big\rangle, \\
I_2 &:= \langle \mathcal{T}(D_p \varphi_{\vec{\tau}}(p^{\vec{\tau}}, q^{\vec{\tau}}), p^{\vec{\tau}})  \eta^+_{\vec{\tau}}, \eta^+_{\vec{\tau}}\rangle
 -\langle \mathcal{T}(-D_q \varphi_{\vec{\tau}}(p^{\vec{\tau}}, q^{\vec{\tau}}), q^{\vec{\tau}})  \eta^-_{\vec{\tau}}, \eta^-_{\vec{\tau}}\rangle.
\end{align*}
The Theorem of Sums (\cite[Theorem 3.2]{CIL:UGTVS} gives us  
\begin{eqnarray*}
I_1 & = & 
 \langle X^{\vec{\tau}}A^T(p^{\vec{\tau}})\eta^+_{\vec{\tau}},A^T(p^{\vec{\tau}})\eta^+_{\vec{\tau}}\rangle-
  \langle Y^{\vec{\tau}}A^T(q^{\vec{\tau}})\eta^-_{\vec{\tau}},A^T(q^{\vec{\tau}})\eta^-_{\vec{\tau}}\rangle\\
  & \leq & \Bigg\langle C \begin{pmatrix} A^T(p^{\vec{\tau}})\eta^+_{\vec{\tau}} \\ A^T(q^{\vec{\tau}})\eta^-_{\vec{\tau}})\end{pmatrix}, \begin{pmatrix} A^T(p^{\vec{\tau}})\eta^+_{\vec{\tau}} \\ A^T(q^{\vec{\tau}})\eta^-_{\vec{\tau}})\end{pmatrix} \Bigg\rangle
\end{eqnarray*}
where
\begin{align*}
C& = \begin{pmatrix} \tau_1+2\delta\tau_1^2 & 0 & 0&-(\tau_1+2\delta\tau_1^2)&0&0\\  
0 & \tau_2+2\delta\tau_2^2 &0&0&-(\tau_2+2\delta\tau_2^2)&0\\
0&0&\tau_3+2\delta\tau_3^2&0&0&-(\tau_3+2\delta\tau_3^2)\\
-( \tau_1+2\delta\tau_1^2) & 0 & 0&\tau_1+2\delta\tau_1^2&0&0\\  
0 & -(\tau_2+2\delta\tau_2^2) &0&0&\tau_2+2\delta\tau_2^2&0\\
0&0&-(\tau_3+2\delta\tau_3^2)&0&0&\tau_3+2\delta\tau_3^2
\end{pmatrix}\\
&=\begin{pmatrix}\mathfrak{B}&-\mathfrak{B}\\-\mathfrak{B}&\mathfrak{B}\end{pmatrix}_.
\end{align*}
We then have 
\begin{equation}\label{I1est}
I_1 \le \Big\langle \mathfrak{B} \cdot \bigg(A^T(p^{\vec{\tau}})\eta^+_{\vec{\tau}} - A^T(q^{\vec{\tau}})\eta^-_{\vec{\tau}}\bigg), \bigg(A^T(p^{\vec{\tau}})\eta^+_{\vec{\tau}} - A^T(q^{\vec{\tau}})\eta^-_{\vec{\tau}}\bigg)\Big\rangle.
\end{equation}

Using Equation \eqref{etadef} and the Twisting Lemma \ref{twlemma}, we have
$A^T(p^{\vec{\tau}})\eta^+_{\vec{\tau}} $ is given by 
\[
\begin{pmatrix}
\tau_1(x^{\vec{\tau}}_1-y^{\vec{\tau}}_1)\\
\mbox{} \\
\tau_2(x^{\vec{\tau}}_2-y^{\vec{\tau}}_2)+f(x_1^{\vec{\tau}})\tau_3(x^{\vec{\tau}}_3-y^{\vec{\tau}}_3)\\
\mbox{} \\
\tau_2(x^{\vec{\tau}}_2-y^{\vec{\tau}}_2)+f(x_1^{\vec{\tau}})\tau_3(x^{\vec{\tau}}_3-y^{\vec{\tau}}_3)+f(x_1^{\vec{\tau}})\tau_2(x^{\vec{\tau}}_2-y^{\vec{\tau}}_2)+f(x_1^{\vec{\tau}})^2\tau_3(x^{\vec{\tau}}_3-y^{\vec{\tau}}_3)
\end{pmatrix} 
\]
and $A^T(q^{\vec{\tau}})\eta^-_{\vec{\tau}}$ is given by
\[
\begin{pmatrix}
\tau_1(x^{\vec{\tau}}_1-y^{\vec{\tau}}_1)\\
\mbox{} \\
\tau_2(x^{\vec{\tau}}_2-y^{\vec{\tau}}_2)+f(y_1^{\vec{\tau}})\tau_3(x^{\vec{\tau}}_3-y^{\vec{\tau}}_3)\\
\mbox{} \\
\tau_2(x^{\vec{\tau}}_2-y^{\vec{\tau}}_2)+f(y_1^{\vec{\tau}})\tau_3(x^{\vec{\tau}}_3-y^{\vec{\tau}}_3)+f(y_1^{\vec{\tau}})\tau_2(x^{\vec{\tau}}_2-y^{\vec{\tau}}_2)+f(y_1^{\vec{\tau}})^2\tau_3(x^{\vec{\tau}}_3-y^{\vec{\tau}}_3)
\end{pmatrix}_.
\]
Thus, $A^T(p^{\vec{\tau}})\eta^+_{\vec{\tau}}-A^T(q^{\vec{\tau}})\eta^-_{\vec{\tau}}$ is given by
\[
\begin{pmatrix}
0\\
\mbox{} \\
(f(x_1^{\vec{\tau}})-f(y_1^{\vec{\tau}}))\tau_3(x^{\vec{\tau}}_3-y^{\vec{\tau}}_3)\\
\mbox{} \\
(f(x_1^{\vec{\tau}})-f(y_1^{\vec{\tau}}))\big(\tau_2(x^{\vec{\tau}}_2-y^{\vec{\tau}}_2)-\tau_3(x^{\vec{\tau}}_3-y^{\vec{\tau}}_3)\big)+(f(x_1^{\vec{\tau}})^2-f(y_1^{\vec{\tau}})^2)\tau_3(x^{\vec{\tau}}_3-y^{\vec{\tau}}_3)
\end{pmatrix}_.
\]
Equation \eqref{I1est} can then be rewritten as 
\begin{eqnarray*}
I_1 & \le & \big(\tau_2+2\delta\tau_2^2\big)(f(x_1^{\vec{\tau}})-f(y_1^{\vec{\tau}}))\tau_3(x^{\vec{\tau}}_3-y^{\vec{\tau}}_3) \\
 & & \mbox{}+\big(\tau_3+2\delta\tau_3^2\big)\Bigg((f(x_1^{\vec{\tau}})-f(y_1^{\vec{\tau}}))\big(\tau_2(x^{\vec{\tau}}_2-y^{\vec{\tau}}_2)-\tau_3(x^{\vec{\tau}}_3-y^{\vec{\tau}}_3)\big) \\
  & &\mbox{}+(f(x_1^{\vec{\tau}})^2-f(y_1^{\vec{\tau}})^2)\tau_3(x^{\vec{\tau}}_3-y^{\vec{\tau}}_3)\Bigg).
\end{eqnarray*}
This leads us to 
\begin{eqnarray*}
\lim_{\tau_1 \to \infty}I_1 & \le & \big(\tau_2+2\delta\tau_2^2\big)(f(x_1^0)-f(x_1^0))\tau_3(x^{\vec{\tau}}_3-y^{\vec{\tau}}_3) \\
 & & \mbox{}+\big(\tau_3+2\delta\tau_3^2\big)\Bigg((f(x^0_1)-f(x_1^0))\big(\tau_2(x^{\vec{\tau}}_2-y^{\vec{\tau}}_2)-\tau_3(x^{\vec{\tau}}_3-y^{\vec{\tau}}_3)\big) \\
  & &\mbox{}+(f(x_1^0)^2-f(x_1^0)^2)\tau_3(x^{\vec{\tau}}_3-y^{\vec{\tau}}_3)\Bigg)\\
  & = & 0+0+0.
\end{eqnarray*}
We then conclude
\begin{equation*}
\lim_{\tau_3 \to \infty} \lim_{\tau_2 \to \infty}\lim_{\tau_1 \to \infty} I_1 = 0.
\end{equation*}

To estimate the term $I_2$, we note that by the Twisting Lemma (Lemma \ref{twlemma}), we have
\begin{eqnarray*}
\mathcal{T}(D_p \varphi_{\vec{\tau}}(p^{\vec{\tau}}, q^{\vec{\tau}}), p^{\vec{\tau}}) & = &
\begin{pmatrix}
0 & \frac{1}{2}\tau_3(x^{\vec{\tau}}_3-y^{\vec{\tau}}_3)f'(x^{\vec{\tau}}_1) \\
 \frac{1}{2}\tau_3(x^{\vec{\tau}}_3-y^{\vec{\tau}}_3)f'(x^{\vec{\tau}}_1) & 0 
\end{pmatrix} \\
\mbox{} \\
\textmd{and \ \ }
\mathcal{T}(-D_q \varphi_{\vec{\tau}}(p^{\vec{\tau}}, q^{\vec{\tau}}), q^{\vec{\tau}}) & = &
\begin{pmatrix}
0 & \frac{1}{2}\tau_3(x^{\vec{\tau}}_3-y^{\vec{\tau}}_3)f'(y^{\vec{\tau}}_1) \\
 \frac{1}{2}\tau_3(x^{\vec{\tau}}_3-y^{\vec{\tau}}_3)f'(y^{\vec{\tau}}_1) & 0 
\end{pmatrix}_.
\end{eqnarray*}

Then by  Equation \eqref{etadef} we have
\begin{eqnarray*}
I_2 & = & \frac{1}{2}\tau_3(x^{\vec{\tau}}_3-y^{\vec{\tau}}_3)f'(x^{\vec{\tau}}_1)\tau_1(x^{\vec{\tau}}_1-y^{\vec{\tau}}_1)  \Big(\tau_2(x^{\vec{\tau}}_2-y^{\vec{\tau}}_2) +f(x_1^{\vec{\tau}})\tau_3(x^{\vec{\tau}}_3-y^{\vec{\tau}}_3)\Big) \\
&  & \mbox{}- \frac{1}{2}\tau_3(x^{\vec{\tau}}_3-y^{\vec{\tau}}_3)f'(y^{\vec{\tau}}_1)\tau_1(x^{\vec{\tau}}_1-y^{\vec{\tau}}_1)  \Big(\tau_2(x^{\vec{\tau}}_2-y^{\vec{\tau}}_2)+f(y_1^{\vec{\tau}})\tau_3(x^{\vec{\tau}}_3-y^{\vec{\tau}}_3)\Big) \\
 & = & \frac{1}{2}\tau_3(x^{\vec{\tau}}_3-y^{\vec{\tau}}_3)\tau_2(x^{\vec{\tau}}_2-y^{\vec{\tau}}_2)\tau_1(x^{\vec{\tau}}_1-y^{\vec{\tau}}_1)(f'(x^{\vec{\tau}}_1)-f'(y^{\vec{\tau}}_1))\\
  & & \mbox{}+\frac{1}{2}\tau^2_3(x^{\vec{\tau}}_3-y^{\vec{\tau}}_3)^2\tau_1(x^{\vec{\tau}}_1-y^{\vec{\tau}}_1)\bigg(f'(x^{\vec{\tau}}_1)f(x^{\vec{\tau}}_1)-f'(y^{\vec{\tau}}_1)f(y^{\vec{\tau}}_1)\bigg)
\end{eqnarray*}
Using Equation \eqref{7.2} we are then able to compute 
\begin{eqnarray*}
\lim_{\tau_1\to\infty} I_2 & = & 
K\tau_3(x^{\vec{\tau}}_3-y^{\vec{\tau}}_3)\tau_2(x^{\vec{\tau}}_2-y^{\vec{\tau}}_2)(f'(x^0_1)-f'(x^0_1))\\
  & & \mbox{}+K\tau^2_3(x^{\vec{\tau}}_3-y^{\vec{\tau}}_3)^2\bigg(f'(x^0_1)f(x^0_1)-f'(x^0_1)f(x^0_1)\bigg)\\
  & = & 0+0 
\end{eqnarray*}
Part D then follows. 
\end{proof}
Using Lemma \ref{inflemma}, we will now prove a comparison principle for viscosity solutions to the infinite Laplace equation, also called infinite harmonic functions, by employing the Jensen auxiliary functions. \cite{Je:ULE}  Namely, we consider for $\varepsilon > 0$, 
\begin{eqnarray*}
F_\varepsilon(\nabla_0u,(D^2u)^\star) & = & 
 \min\{\|\nabla_0u\|^2-\varepsilon^2,
-\ip{(D^2u)^\star \nabla_0u}{\nabla_0u}\} \\
G_\varepsilon(\nabla_0u,(D^2u)^\star) & = & 
 \max\{\varepsilon^2-\|D\nabla_0u\|^2,
-\ip{(D^2u)^\star \nabla_0u}{\nabla_0u}\} \\
\textmd{and\ \ } \mathcal{F}_\infty & = & -\ip{(D^2u)^\star \nabla_0u}{\nabla_0u}.
\end{eqnarray*}
We observe that a viscosity solution to $F_\varepsilon=0$ is a viscosity infinite supersolution to $\mathcal{F}_\infty=0$ and a viscosity solution to $G_\varepsilon=0$ is viscosity infinite subsolution to $\mathcal{F}_\infty$.
\begin{theorem}\label{fthm}
Let $u$ be a viscosity subsolution and $v$ a
viscosity supersolution to
\[F_\varepsilon(\nabla_0f(p),(D^2f(p))^\star)=0\] in a bounded domain $\Omega$. In addition, assume one of $u, v$ is locally Lipschitz.  Then, if $u\leq v$ on $\partial \Omega$ then $u\leq v$ in $\overline{\Omega}$. 
\end{theorem}
\begin{proof}
Suppose $u(p_0)>v(p_0)$ for some $p_0\in\Omega$.  Using the argument from \cite{B:HG}, we may assume $v$ is a strict supersolution so that using the Iterated Maximum Principle (Lemma \ref{IMP}) and Lemma \ref{inflemma} Parts A and C, we have 
\begin{eqnarray*}
0<\mu(p) & \leq & F_\varepsilon(\eta^-_{\vec{\tau}},\mathcal{Y}_{\vec{\tau}})\\
0 & \geq & F_\varepsilon(\eta^+_{\vec{\tau}},\mathcal{X}_{\vec{\tau}})
\end{eqnarray*}
for an appropriate function $\mu(p)$. Subtracting these two inequalities yields 
\[0<\mu(p) \leq \max\{\|\eta^-_{\vec{\tau}}\|^2-\|\eta^+_{\vec{\tau}}\|^2,
\ip{\mathcal{X}_{\vec{\tau}}\eta^+_{\vec{\tau}}}{\eta^+_{\vec{\tau}}}
-\ip{\mathcal{Y}_{\vec{\tau}}\eta^-_{\vec{\tau}}}{\eta^-_{\vec{\tau}}}\}.\]
The corollary follows from Lemma \ref{inflemma} Parts C and D. 
\end{proof}
This theorem has an immediate corollary.  
\begin{corollary}\label{gcor}
Let $u$ be a viscosity subsolution and $v$ a
viscosity supersolution to
\[G_\varepsilon(\nabla_0f(p),(D^2f(p))^\star)=0\] in a bounded domain $\Omega$. In addition, assume one of $u, v$ is locally Lipschitz.  Then, if $u\leq v$ on $\partial \Omega$ then $u\leq v$ in $\overline{\Omega}$. 
\end{corollary}

We now state a lemma  giving an estimate on the solutions as
$\varepsilon \rightarrow 0$. 
The proof is similar to the Euclidean version proved by Jensen \cite{Je:ULE} and is omitted.  
\begin{lemma}\label{lemmaunique}
Let $g \in W^{1,\infty}(\Omega)$.
Let $ u^{\varepsilon}$ and $u_{\varepsilon}$ be solutions to 
$ F_{\varepsilon}=0$ and $G_{\varepsilon}=0$,
 respectively, equal to $g$ on $\partial \Omega$.   Given $\delta >0$, there exists an $\varepsilon > 0$ so that
\[u_{\varepsilon} \leq u^{\varepsilon} \leq u_{\varepsilon} + \delta.\]
\end{lemma}

We then combine Theorem \ref{fthm}, Corollary \ref{gcor}, and
Lemma \ref{lemmaunique} to obtain the comparison principle for viscosity infinite
harmonic functions. The proof is a standard argument and omitted. \cite{B:HG}  
\begin{theorem}\label{infharmunique}
Let $u$ be a viscosity infinite subharmonic function and $v$ be a viscosity infinite superharmonic function in a domain $\Omega$ such that
 if $p \in \partial \Omega,$
$$\limsup_{q\rightarrow p}u(q) \leq \limsup_{q\rightarrow p}v(q)$$ where both
sides are not $-\infty$ or $+\infty$ simultaneously.  Then $u \leq v$ in $\Omega$.
\end{theorem}
\begin{corollary}
    Let $\Omega$ be a bounded domain in $\mathbb{M}$ and let $g$ be a continuous function on $\partial \Omega$. Then the Dirichlet problem 
    \begin{eqnarray*}
\left\{ \begin{array}{cl}
\Delta_{\infty}u  =  0  & \textmd{in\ } \Omega \\
u = g & \textmd{on\ } \partial \Omega
\end{array} \right.
\end{eqnarray*}
has a unique viscosity solution $u$. 
\end{corollary}

\end{document}